\documentclass[conference]{IEEEtran}

\usepackage[pdftex]{graphicx}
\DeclareGraphicsExtensions{.pdf,.jpeg,.png}
\usepackage{algorithm}
\usepackage{algpseudocode}

\usepackage{etoolbox}
\usepackage{tabu}
\usepackage{setspace}
\usepackage{makecell}
\usepackage{epsfig}
\usepackage[usenames, dvipsnames]{color}
\usepackage[roman]{parnotes}

\usepackage{amsthm,amsmath,amssymb}
\usepackage{amsmath}
\usepackage{mathtools}
\usepackage{breqn}
\usepackage{amsfonts}
\usepackage{amssymb}
\usepackage{amsthm}
\usepackage{mathrsfs}
\usepackage{xparse}
\usepackage{multirow}
\usepackage{chngcntr}
\usepackage{dblfloatfix}    
\usepackage{soul} 
\usepackage{cite}

\makeatletter
\def\ps@IEEEtitlepagestyle{%
  \def\@oddhead{\mycopyrightnotice}%
  \def\@oddfoot{\hbox{}\@IEEEheaderstyle\leftmark\hfil\thepage}\relax
  \def\@evenhead{\@IEEEheaderstyle\thepage\hfil\leftmark\hbox{}}\relax
  \def\@evenfoot{}%
}
\def\mycopyrightnotice{%
  \begin{minipage}{\textwidth}
  \scriptsize
  Copyright~\copyright~2020 IEEE. Personal use of this material is permitted. Permission from IEEE must be obtained for all other uses, in any current or future media, including reprinting/republishing this material for advertising or promotional purposes, creating new collective works, for resale or redistribution to servers or lists, or reuse of any copyrighted component of this work in other works by sending a request to pubs-permissions@ieee.org. Accepted and published in 2020 IEEE Latin-American Conference on Communications. Citation information: DOI: 10.1109/LATINCOM50620.2020.9282342.
  \end{minipage}
}
\makeatother

\makeatletter
\newcommand*{\algrule}[1][\algorithmicindent]{%
  \makebox[#1][l]{%
    \hspace*{.2em}
    \vrule height .75\baselineskip depth .25\baselineskip
  }
}

\newcount\ALG@printindent@tempcnta
\def\ALG@printindent{%
    \ifnum \theALG@nested>0
    \ifx\ALG@text\ALG@x@notext
    \else
    \unskip
    \ALG@printindent@tempcnta=1
    \loop
    \algrule[\csname ALG@ind@\the\ALG@printindent@tempcnta\endcsname]%
    \advance \ALG@printindent@tempcnta 1
    \ifnum \ALG@printindent@tempcnta<\numexpr\theALG@nested+1\relax
    \repeat
    \fi
    \fi
}
\patchcmd{\ALG@doentity}{\noindent\hskip\ALG@tlm}{\ALG@printindent}{}{\errmessage{failed to patch}}
\patchcmd{\ALG@doentity}{\item[]\nointerlineskip}{}{}{} 
\makeatother

\allowdisplaybreaks

\IEEEoverridecommandlockouts \IEEEpubid{\makebox[\columnwidth]{
978-1-7281-8903-1/20/\$31.00~\copyright{}2020 IEEE \hfill}
\hspace{\columnsep}\makebox[\columnwidth]{ }}

\begin{document}
%
\title{A Low-Complexity Multi-Survivor Dynamic Programming for Constrained Discrete Optimization}

\author{\IEEEauthorblockN{I. Zakir Ahmed  and Hamid Sadjadpour\\}
\IEEEauthorblockA{Department of Electrical and Computer Engineering\\
University of California, Santa Cruz\\
}
\and
\IEEEauthorblockN{Shahram Yousefi}
\IEEEauthorblockA{Department of Electrical and Computer Engineering\\
Queen's University, Canada\\
}}


%


\maketitle

\begin{abstract}
Constrained discrete optimization problems are encountered in many areas of communication and machine learning. We consider the case where the objective function satisfies Bellman's optimality principle without the constraints on which we place no conditions. We first show that these problems are a generalization of optimization in constrained Markov decision processes with finite horizon used in reinforcement learning and are NP-Hard. We then present a novel multi-survivor dynamic programming (msDP) algorithm that guarantees optimality at significant computational savings. We demonstrate this by solving 5G quantizer bit allocation and DNA fragment assembly problems. The results are very promising and suggest that msDP can be used for many applications.
\end{abstract}
\IEEEpeerreviewmaketitle

\newcommand{\Xmatrix}{
\begin{bmatrix}
\ddots  & 0     & 0 \\
0  & \frac{1}{{\sigma_i^2}} & 0 \\
0 & 0 & \ddots
\end{bmatrix}}
\newcommand{\Ymatrix}{
\begin{bmatrix}
\ddots  & 0  & 0 \\
0  & \frac{f(b_i)l_i}{\big(1-f(b_i)\big) \sigma_i^2} & 0 \\
0 & 0 & \ddots
\end{bmatrix}}
\newcommand{\Zmatrix}{
\begin{bmatrix}
\ddots  & 0  & 0 \\
0  & \frac{\sigma_i^2}{\sigma_n^2 + \frac{f(b_i)l_i}{\big(1-f(b_i)\big)}} & 0 \\
0 & 0 & \ddots
\end{bmatrix}}
\newcommand{\ZMmatrix}{
\begin{bmatrix}
\ddots  & 0  & 0 \\
0  & \frac{\sigma_i^2}{\sigma_n^2 + \frac{f(b_i)l_i}{\big(1-f(b_i)\big)}} + \frac{1}{p} & 0 \\
0 & 0 & \ddots
\end{bmatrix}}
\newcommand{\Fmatrix}{
\begin{bmatrix}
\ddots  & 0  & 0 \\
0  & 10 & 0 \\
0 & 0 & \ddots 
\end{bmatrix}}
\newcommand{\InvFmatrix}{
\begin{bmatrix}
\ddots  & 0 & 0 \\
0  & f(b_i)\big(1-f(b_i)\big)l_i & 0 \\
0 & 0 & \ddots
\end{bmatrix}}
\section{Introduction}
\label{Intro}
We define a class of problems $H$ that are constrained discrete optimization problems with the objective function (OF) satisfying the principle of optimality (PO) without considering the constraints \cite{Bellman}. The constraint functions are assumed to be neither convex nor linear in their decision variables. Nor are the constraints required to satisfy the linear independence constraint qualification (LICQ) \cite{Nocedal}. These problems, in general, are NP-Hard \cite{NPHard}. Solving $H$ to optimality with reduced computational resources is highly desirable. Popular methods like Dynamic Programming (DP), Branch and Bound (BB), and Lagrangian relaxation methods can solve $H$ to optimality \cite{DynP1,DynP2,Bbound, Kolman, Fisher} as long as we place either linearity or quadratic conditions on the OF or constraints \cite{Bbound,LinShu}, while many others do not ensure optimality \cite{Plancher}.\\
\indent Examples of $H$ include resource allocation (power, ADC bits) in massive Multiple-Input Multiple-Output (MaMIMO) systems with constraints \cite{Zakir2}, constrained network optimization problems used in routing protocols \cite{NetApp}, and Reinforcement Learning (RL) applied to Markov decision processes with constraints \cite{RLbook}, to name a few.\\

\noindent The contributions of this paper are as follows
\begin{itemize}
\item we propose a novel multi-survivor Dynamic Programming (msDP) algorithm to solve $H$ optimally at significantly reduced computational cost, and, 
\item we show that Constrained Markov Decision Processes (CMDP) with finite horizons belong to $H$.
\end{itemize}

\textit{Notations:}
The column vectors are represented as boldface small letters. The superscripts $T$ denote transpose. The terms $\mathbb{I}$ and $\mathbb{R}$, indicate space of integers and real numbers respectively. We represent a vector $\bold{x} = [x_1, x_2,\cdots,x_N]^T$ whose entries are constants up to $x_i$ for $i \le m$ as $\bold{x}^{m}$ and call such vectors partial, i.e., $\bold{x}^{m} = [a_1, a_2,\cdots, a_m, x_{m+1},\cdots, x_N]^T$ where $a_i$'s are constants. The $x_i$'s belong to a finite set $\mathcal{X}$ with cardinality $M$. If there are multiple of such partial vectors under consideration, we refer to the $k^{\text{th}}$ such vector as $\bold{x}_{k}^{m}$ and its elements as $\bold{x}_{k}^{m} = \big[a_{1}^{k},a_{2}^k,\cdots,a_{m}^k, x_{m+1}^k,x_{m+2}^k,\cdots,x_N^k \big]^T$.  The $j^{\text{th}}$ element of the set $\mathcal{X}$ is represented as $\{ \mathcal{X} \}_j$. If the $m^{\text{th}}$ element of $\bold{x}_{k}^{m}$ takes a value $\{ \mathcal{X} \}_j$, we represent such partial vector as $\bold{x}_{k}^{mj}$.\\

\indent The rest of this paper is organized as follows. We define the problem $H$ in Section \ref{Prob}. We first show in Section \ref{Cmdp} that the CMDP framework typically used in RL belongs to $H$. In Section \ref{TF} we describe the theoretical foundations that lead to the proposed msDP algorithm. We describe the proposed algorithm and evaluate its worst-case computational complexity in Section \ref{Algo}. We compare the performance and computational complexity of the proposed method with exhaustive search (ES) and the heuristic algorithm in Section \ref{Sim}, followed by the conclusions in Section \ref{Conc}.

\section{Problem Setup}
\label{Prob}
We define $H$ as stated below, where $\bold{x}^*$ is the optimal solution to $\eqref{cdo_1}$ if it exists. 
\begin{equation}\label{cdo_1}
\begin{split}
&\underbrace{\text{max}}_{\substack{\bold{x}}} f(\bold{x}),\\
\text{ such that }&c_i(\bold{x}) \le \alpha_i;\text{ for }1 \le i \le Q_{I},\\
&h_j(\bold{x}) = \beta_j;\text{ for }1 \le j \le Q_E,
\end{split}
\end{equation}
where $\bold{x} = [x_1, x_2,\cdots,x_N]^T \in \mathbb{R}^N$. The OF $f(\bold{x})$ satisfies the PO without the constraint functions $c_i(\bold{x})$ and $h_j(\bold{x})$, where $\alpha_i, \beta_j \in \mathbb{R}, \forall i,j$ \cite{Bellman}. Thus, it can be written in the form $f(\bold{x}) = \sum_{i=1}^{N} b_i \phi_i(x_i)$ where $b_i \in \mathbb{R}$ are constants, $x_i \in \mathcal{X}$ can only take values from the set $\mathcal{X}$ whose cardinality is $M$, and, $\phi_i:\mathcal{X} \longrightarrow \mathcal{Y},\text{ for } 1\le i \le N$. The mapping $\phi_i$ need not be known in closed form. The constraint functions $c_i(\bold{x})$ and $h_j(\bold{x})$ are not limited to linear mappings in $x_i$, nor convex and need not satisfy LICQ \cite{Nocedal}. The terms $Q_{I}$ and $Q_E$ represent the number of inequality and equality constraints, respectively.

\section{Finite Horizon CMDP as problem $H$}
\label{Cmdp}
A finite-state Markov decision process is a 5-tuple $(X,A,P,r,\mu)$ where $X$ is the finite set of states, $A$ is the finite set of actions, $P:X \times A \times X^{'} \rightarrow [0,1]$ are the state transition probabilities $p_{x,a}(x^{'})$ that a state $x^{'}$ is attained when an action $a \in A$ is taken in state $x$. Both $x, x^{'} \in X$. A reward $r:X \times A \rightarrow \mathbb{R}$ is associated with an $a \in A$ from a state $x \in X$. The map $\mu : X \rightarrow [0,1]$ is the starting state distribution \cite{RL2}.\\ 
\indent A policy $\pi = (D_1, D_2,\cdots,D_{N-1})$ is defined as a set of decisions taken by an agent over states that maximize a certain cumulative reward. A stationary deterministic policy has a decision rule $D_i:X \rightarrow A$ that maps states to actions. In case of randomized policies, the decision rule $D_i$ is a random variable and described using state dependent distributions on possible actions. Hence, $D_i$ is described as $D_i:X \rightarrow \mathcal{P}(A)$ where $\mathcal{P}(A)$ is the probability distribution over actions defined by $q_i(a|x) = \mbox{Prob}[D_i=a|x_i]$ for every action $a \in A$.\\
\indent The cumulative reward is defined as $R = \sum_{i=1}^N \gamma^i r_i$ where $r_i$ is a reward accrued at $i^{\text{th}}$ time instance when the agent uses a policy $\pi$. The discount factor $\gamma \in [0,1]$. We can write the expected sum of discounted rewards as a function of a given policy and an initial state $x$ as $V^{\pi}(x) = \mathbb{E}[R \mid x_1=x] $\cite{RL2}, which can be rewritten as
\begin{equation}\label{cdo_3}
\begin{split}
&V^{\pi}(x) = \mathbb{E}\bigg[ \sum_{i=1}^N \gamma^i r_i\big(x_i,D_i(x_i)\big) \mid x_1=x\bigg],\\
&= \sum_{i=1}^N \gamma^i \hat{r_i}(x_i),\text{ where }\hat{r_i}(x_i) =  \sum_{a \in A} q_i(a|x_i)r_i(x_i,a),
\end{split}
\end{equation}
assuming a random stationary policy. Considering an initial distribution $\mu$ on all the possible starting states on set $X$, we can define the value of the policy as \cite{RL2}
\begin{equation}\label{cdo_5}
\begin{split}
J(\pi) &= \mathbb{E}_{x \sim \mu}\bigg[ V^{\pi}(x)\bigg] = \sum_{x \in X} \mu(x)V^{\pi}(x).\\
&= \sum_{x \in X} \mu(x) \sum_{i=1}^N \gamma^i \hat{r_i}(x_i) = \sum_{i=1}^N \gamma^i \phi_i(x_i),\\
\end{split}
\end{equation}
where $\phi_i(x_i) = \sum_{x \in X} \mu(x)\hat{r_i}(x_i)$, which is obtained by substituting $V^{\pi}(x)$ from \eqref{cdo_3}. We see that with $b_i = \gamma^i$ the OF $J(\pi)$ of \eqref{cdo_5} is similar to the OF $f(\bold{x}) = \sum_{i=1}^{N} b_i \phi_i(x_i)$ of $H$ in \eqref{cdo_1}.\\
\indent For the CMDP case, we have constraint functions that could be either secondary reward functions or penalty functions. The constraint functions $c$ defined as $c:X \times A \rightarrow \mathbb{R}$ such that $c^{\pi}(x) \le d$. The whole optimization problem is written as \cite{RL2}
\begin{equation}\label{cdo_6}
\underbrace{\text{max}}_{\substack{\pi}} J(\pi),\text{ such that }c(\pi) \le d.
\end{equation}
In \eqref{cdo_6}, $c(\pi) = \mathbb{E}[c^{\pi}(x)] = \mathbb{E} \big[ \sum_{i=1}^N \gamma^i c_i(x_i,a_i) \big]$.

\section{Theoretical Foundations of msDP}\label{TF}
In this section, we construct the theoretical foundations for the proposed msDP. We first define the Constraint Satisfaction Function (CSF).
We define a boolean function $A:\bold{x} \rightarrow \{0,1\}$ such that
\begin{equation}\label{cdo_8}
    A(\bold{x})= 
\begin{cases}
    1,& \text{if } \bold{x} \text{ satisfies all the constraints }\\ 
    &c_i(\bold{x}), h_j(\bold{x}) \text{ of H; for all } i,j.\\
    0,              & \text{otherwise.}
\end{cases}
\end{equation}
We also define $A(\cdot)$ on a partially assigned vector $\bold{x}_{k}^{m}$ as $A(\bold{x}_{k}^{m}) = 1$ for $1 \le m < N$ if there exists at least one vector ${\bold{b}_k} = [b_{m+1}^k,b_{m+2}^k,\cdots,b_{N}^k]^T$ such that $b_i \in \mathcal{X}$ for $i > m$ that satisfies all the constraints. This is represented as
\begin{equation}\label{cdo_9}
    A(\bold{x}_{k}^{m})= 
\begin{cases}
    1,& \text{if there exists at least one } \bold{b}_k \text{ defined}\\
    &\text{ above, that satisfies all the constraints }\\ 
    &c_i(\bold{x}_{k}^{m}), h_j(\bold{x}_{k}^{m}) \text{ for all } i,j.\\
    0,              & \text{otherwise.}
\end{cases}
\end{equation}

\subsection{Principle of Optimality}
We define a partially computed OF of $H$ as $\lambda(\bold{x}_{k}^{m}) = \sum_{i=1}^{m} b_i \phi_i(a_i)$ for $1 \le m < N$. It is to be noted that $A(\bold{x}) = A(\bold{x}^N)$ and $\lambda(\bold{x}^{N}) = f(\bold{x})$. Here $\bold{x}^N$ is a completely known vector.\\
\indent Now we show that $H$ does not satisfy the PO. We then introduce a multi-valued function $\mbox{max}_l\{\cdot\}$ that returns $l$ maximum values from its domain. Adopting $\mbox{max}_l\{\cdot\}$, we reformulate $H$ to satisfy the PO.\\

\newtheorem{thrm}{Theorem}
\begin{thrm}\label{thrm1}
If $\bold{x}^*$ is the optimal solution to $H$ described in \eqref{cdo_1} and if there exists at least one infeasible solution $\bold{x}_1$ to $H$ such that $f(\bold{x}_1) > f(\bold{x}^*)$ with constraint satisfaction function being true at some intermediate stage $l$ in the recursion $A(\bold{x}_1^l) > 0$, then it can be shown that $J = \underbrace{\max}_{\bold{x},A(\bold{x}) > 0}f(\bold{x})$ does not obey the PO.
\end{thrm}
\begin{proof}
We say that \eqref{cdo_1} satisfies the PO if we can express it as a value function \cite{Bellman}. That is, given
\begin{equation}\label{cdo_10}
\begin{split}
J &= f(\bold{x}^*) = J(x_1) = \underbrace{\text{max}}_{\substack{\{x_i\}_{i=1}^N,\\A(\bold{x}) > 0}}\Big\{\sum_{i=1}^{N} b_i \phi_i(x_i)\Big\},\text{ then}\\
J(x_i) &= \underbrace{\text{max}}_{\substack{x_i;\\A(\bold{x}^{i})>0}}\Bigg\{ b_i\phi_i(x_i) +  \underbrace{\text{max}}_{\substack{\{x_j\}_{j={i+1}}^N;\\{A(\bold{x}^{i+1}) > 0}}} \sum_{j={i+1}}^{N} b_j \phi_i(x_j) \Bigg\},\\
J(x_i) &= \underbrace{\text{max}}_{\substack x_i;A(\bold{x}^{i})>0}\Bigg\{ b_i\phi_i(x_i) +  J(x_{i+1}) \Bigg\}.
\end{split}
\end{equation}
\indent Let $\bold{x}_1$ be an infeasible solution to $H$ such that 
\begin{equation}\label{cdo_10a}
\begin{split}
f(\bold{x}_1) &> f(\bold{x}^{*}),\text{ where}\\
A(\bold{x}_1) &= 0\text{ and }A(\bold{x}^*) > 0.
\end{split}
\end{equation}
However, at some intermediate stage $l < N$ in the recursion, the condition depicted below would arise as a consequence of no conditions placed on the constraints \cite{Nocedal, Boyd}.\\
As $\lambda(\bold{x}_1^l) > \lambda(\bold{x}^{*l}),\text{ such that }A(\bold{x}_1^l) > 0\text{ and }A(\bold{x}^{*l}) > 0: \lambda(\bold{x}_1^l) = \underbrace{\text{max}}_{\substack x_l;A(\bold{x}^{l})>0}\bigg\{ \lambda(\bold{x}_1^l), \lambda(\bold{x}^{*l}), \cdots \bigg\}$.
Thus, we can write the recursion at stage $l$ in \eqref{cdo_10} as $J = \lambda(\bold{x}_1^l) +  J(x_{l+1})$.
As a result, we see that the optimal solution $\bold{x}^{*l}$ is dropped against $\bold{x}_1^{l}$ at stage $l$. However, in the subsequent stage $r$, where $l < r \le N$, another candidate solution $\bold{x}_2^r$ is picked against $\bold{x}_1^r$ because the infeasibility of $\bold{x}_1^{r}$ shows up as $A(\bold{x}_2^r) > 0,\text{ and }A(\bold{x}_1^r) = 0$.
Consequently, $\bold{x}_1^{r}$ is dropped and we have the recursion equation at stage $r \le N$ as
\begin{equation}\label{cdo_14}
\begin{split}
J &= \lambda(\bold{x}_2^r) +  J(x_{r+1}).
\end{split}
\end{equation}
Extending the result to stage $N$, we can see that 
\begin{equation}\label{cdo_14}
\begin{split}
J &= f(\bold{x}_2) = \lambda(\bold{x}_2^N) = \underbrace{\text{max}}_{\substack{\bold{x};A(\bold{x}) > 0}}\sum_{i=1}^{N} b_i\phi_i(x_i).
\end{split}
\end{equation}
However, by definition we have
\begin{equation}\label{cdo_14}
J = f(\bold{x}^*) = \lambda(\bold{x}^{*N}) = \underbrace{\text{max}}_{\substack{\bold{x};A(\bold{x}) > 0}}\sum_{i=1}^{N} b_i\phi_i(x_i),
\end{equation}
where $\bold{x}^{*} \ne \bold{x}_2$, hence, $H$ defined in \eqref{cdo_1} does not satisfy the PO.
\end{proof}
We now define a multi-valued function $\mbox{max}_{N_e}:\mathbb{R}^m \rightarrow \mathbb{R}^{N_e}$ such that it returns at-most $N_e$ maximum values over a combination of $m$ variables in its domain. We now rewrite $H$ in \eqref{cdo_1} using $\Gamma$ as defined below
\begin{equation}\label{cdo_15}
\begin{split}
&\underbrace{\text{max}}_{\substack {\bold{x}};A(\bold{x})>0} f(\bold{x}) = J = \text{max}\Big\{ \Gamma(x_1) \Big\},\text{ where}\\
&\Gamma(x_1) = \boldsymbol{\lambda}(\bold{\vec{x}}) = \underbrace{\mbox{max}_{N_e}}_{\substack \{x_j\}_{j=1}^{N};A(\bold{\vec{x}})>0}\Big\{\sum_{i=1}^{N} b_i \phi_i(x_i)\Big\},
\end{split}
\end{equation}
where $\boldsymbol{\lambda}(\bold{\vec{x}})$ is a vector of $N_e$ maximum values. That is, $\boldsymbol{\lambda}(\bold{\vec{x}}) = [\lambda(\bold{x}_1), \lambda(\bold{x}_2),\cdots,\lambda(\bold{x}_{N_e})]^T$ that corresponds to the $N_e$ solutions $\bold{\vec{x}} = [\bold{x}_1,\bold{x}_2,\cdots,\bold{x}_{N_e}]^T$. It is also to be noted that the CSF $A(\bold{\vec{x}})>0$ is true iff $A(\bold{x}_i)>0$ is satisfied for all  $\bold{x}_i \in \bold{\vec{x}}$.\\
\indent We also define the partially evaluated objective at stage $m \le N$ as $\boldsymbol{\lambda}(\bold{\vec{x}}^m) = [\lambda(\bold{x}_1^m), \lambda(\bold{x}_2^m),\cdots,\lambda(\bold{x}_{N_e}^m)]^T = \underbrace{\text{max}_{N_e}}_{\substack \{x_j\}_{j=1}^{m};A(\bold{\vec{x}}^m)>0}\Big\{\sum_{i=1}^{m} b_i \phi_i(x_i)\Big\}$. As before, the CSF  $A(\bold{\vec{x}}^m)>0$ is true iff $A(\bold{x}_i^m)>0$ is satisfied for all  $\bold{x}_i^m \in \bold{\vec{x}}^m$.\\
Given the optimal solution $\bold{x}^*$, we say that the value function $\Gamma(\cdot)$ in \eqref{cdo_15} satisfies the PO if we can write 
\begin{equation}\label{cdo_16}
\begin{split}
\Gamma(x_i) &= \underbrace{\mbox{max}_{N_e}}_{\substack x_i;A(\bold{\vec{x}}^{i})>0}\Big\{b_i\phi_i(x_i) + \Gamma(x_{i+1})\Big\},\text{ where}\\
J &= f(\bold{x}^*) = \text{max}\Big\{ \Gamma(x_1) \Big\}.
\end{split}
\end{equation}
That is, the optimal solution is maintained at every stage of the recursion in $\Gamma(x_i)$.
\begin{thrm}\label{thrm2}
If $\bold{x}^*$ is the optimal solution to $H$ described in \eqref{cdo_1}, then with an appropriate choice of $N_e$ that is dependent on the constraint satisfaction function $A(\cdot)$, it can be shown that $\Gamma(x_1) = \underbrace{\max_{N_e}}_{\substack \{x_j\}_{j=1}^{N};A(\bold{\vec{x}})>0}\Big\{\sum_{i=1}^{N} b_i \phi_i(x_i)\Big\}$ always satisfies the PO.
\end{thrm}
\begin{proof}
Let $\bold{x}_1$ be an infeasible solution to $H$ as described in \eqref{cdo_10a} in Theorem \ref{thrm1}. Let the PO \eqref{cdo_16} hold true; we can then write the same at recursion $l \le N$ as
\begin{equation}\label{cdo_17}
\begin{split}
\Gamma(x_1) &= \underbrace{\mbox{max}_{N_e}}_{\substack x_l;A(\bold{\vec{x}}^l)>0}\Big\{b_l \phi_l(x_l)  + \Gamma(x_{l+1})\Big\},\\
&= \boldsymbol{\lambda}(\bold{\vec{x}}^l) + \Gamma(x_{l+1}).
\end{split}
\end{equation}
We now let $N_e$ be large enough such that $\boldsymbol{\lambda}(\bold{\vec{x}}^l) = [\lambda(\bold{x}_1^l), \lambda(\bold{x}^{*l}),\cdots]^T$. That is $(\bold{x}_1^l, \bold{x}^{*l}) \in \bold{\vec{x}}^l$.\\
\indent However, at some later stage $r$ where $l < r \le N$, we can write the recursion as
\begin{equation}\label{cdo_18}
\begin{split}
\Gamma(x_1) &= \underbrace{\mbox{max}_{N_e}}_{\substack x_r;A(\bold{\vec{x}}^r)>0}\Big\{b_r \phi_l(x_r)  + \Gamma(x_{r+1})\Big\},\\
&= \boldsymbol{\lambda}(\bold{\vec{x}}^r) + \Gamma(x_{r+1}).
\end{split}
\end{equation}
Now we note that $\lambda(\bold{x}_2^r) \in \boldsymbol{\lambda}(\bold{\vec{x}}^r)$ and $\lambda(\bold{x}_1^r) \not\in \boldsymbol{\lambda}(\bold{\vec{x}}^l)$. This is because, at this stage in the recursion, the infeasibility of the candidate solution $\bold{x}_1^r$ becomes known as indicated by $A(\bold{x}_1^r) = 0$. Thus the solution $\bold{x}_1^r$ is not part of $\bold{\vec{x}}^r$.\\
With the appropriate choice of $N_e$, it can be ensured that $(\bold{x}_2^r, \bold{x}^{*r}) \in \bold{\vec{x}}^r$. The determination of $N_e$ is discussed in Lemmas \ref{lemma3} and \ref{lemma4} in section \ref{TF}. Extending $r$ to $N$, we have $(\bold{x}_2^N, \bold{x}^{*N}) \in \bold{\vec{x}}^N$. Hence, we can write
\begin{equation}\label{cdo_19}
\begin{split}
\Gamma(x_1) &= \boldsymbol{\lambda}(\bold{\vec{x}}^N) =  \underbrace{\mbox{max}_{N_e}}_{\substack \{x_j\}_{j=1}^{N};A(\bold{\vec{x}}^N)>0}\Big\{\sum_{i=1}^{N} b_i \phi_i(x_i)\Big\},\\
&\text{with }\bold{x}^{*} \in \bold{\vec{x}}^N;\text{ Thus, }f(\bold{x}^*) = \text{max}\Big\{ \Gamma(x_1) \Big\}.
\end{split}
\end{equation}
Thus, with an appropriate choice of $N_e$, we see that the PO for \eqref{cdo_15} is always satisfied. 
\end{proof}

\subsection{Trellis for H}\label{graph}
We define a graph $G(V,E,A)$ with nodes $V$, edges $E$, and edge label $A$ as
\begin{equation}\label{mlba_4}
\begin{split}
V_i &= \Big\{ x_{ij} \mid x_{ij} = \{ \mathcal{X} \}_j; 1 \le i \le N; 1 \le j \le M \Big\} \cup S \cup F,\\
E &=\Big\{(x_{ij},x_{{i+1}m}) \mid x_{ij} = \{ \mathcal{X} \}_j; x_{{i+1}m} = \{ \mathcal{X} \}_m;\\
&1 \le j,m \le M; 1 \le i \le N-1; (S,x_{1j}),  (x_{Nm},F) \Big\},\\
A &= \Big\{ \psi_i(x_{i+1} \mid x_{i}) =  \phi_i(x_i) \leftarrow (x_{i},x_{i+1}) \mid \forall x_i \in \mathcal{X};\\
&1 \le i \le N; \phi_N(x_N) \leftarrow (x_N,F) \mid \forall x_N \in \mathcal{X}; \\
& 0  \leftarrow (S,x_1) \mid \forall x_1 \in \mathcal{X} \Big\}.
\end{split}
\end{equation}
The nodes $S$ and $F$ are the starting and terminal nodes, respectively. The optimal solution $\bold{x}^* = [x_{1j},x_{2j},\cdots,x_{Nj}]^T$ can be thought of as a walk through this graph starting at node $S$ and ending at the destination node $F$ passing through the nodes $x_{ij} \in V_i$ from stage $i=1$ to $i=N$, along the edges $(x_{ij},x_{{i+1}m}) \in E$ accruing a reward $\psi_i(x_{{i+1}m} \mid x_{ij}) = \phi_i(x_{ij}) \in A$ that results in the maximum total accumulated reward $f(\bold{x}^*)$ such that all the constraints are satisfied. That is  $A(\bold{x}^*) > 0$. It can also be noted that this graph $G(V,E,A)$ can be qualified as a trellis  $T(V,E,A)$ \cite{Kolman}. An example $T(V,E,A)$ with $M=4$ is shown in Fig \ref{Fig2_mlba}.\\
\begin{figure*}[http]
\centering
\resizebox{.7\textwidth}{!}{
\input{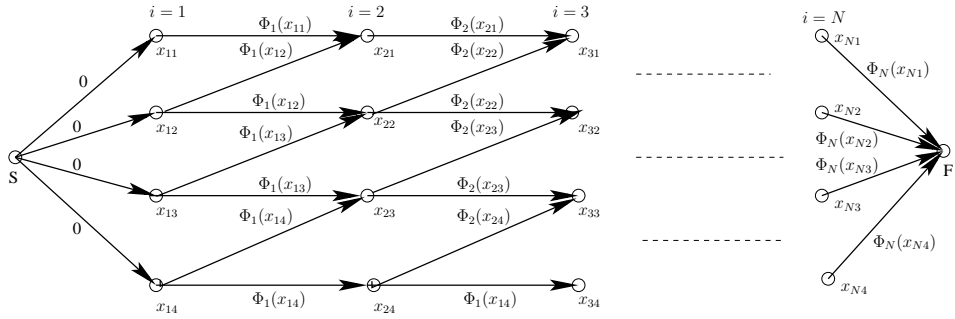}
}
\caption{An example trellis $T(V,E,A)$ for $M=4$.}
\label{Fig2_mlba}
\end{figure*}
\indent A DP approach can not be used to obtain $\bold{x}^*$ as shown using Theorem \ref{thrm1}. We note from Theorem \ref{thrm2} that the PO is satisfied by using the multi-valued function $\mbox{max}_{N_e}$. Thus, maintaining multiple survivor paths $(N_e)$ at each stage of the trellis ensures optimality if $N_e$ is picked properly. An appropriate choice of  $N_e$ is discussed in the next subsection. Thus the DP algorithm is modified to perform Add, Compare, and Multiple Select (ACMS) as basic operation at each node in the trellis as opposed to Add, Compare, and Select (ACS). The algorithm is described in section \ref{Algo}.

\subsection{Determination of $N_e$}\label{SecNe}
In the Lemmas below, we establish a relationship between the number of survivors $N_e$ and CSF $A(\cdot)$ to ensure optimality. Due to the page limitation, the proofs are omitted.
\newtheorem{lemm}{Lemma}
\begin{lemm}\label{lemma3}
The number of survivor paths (or partial solutions) that need to be maintained at a given node $x_{ij}$ in trellis $T(V,A,E)$ to obtain optimal solution to $H$ is dependent on the constraint satisfaction function $A(\bold{x}^{ij})$ at the node $x_{ij}$ in the trellis. It can be shown that 
\end{lemm}
\begin{equation}\label{cdo_lem1}
B(x_{ij}) \triangleq \sum_{k=1}^{L_{ij}} A(\bold{x}^{ij}_k),
\end{equation}
where $L_{ij}$ is the number of paths incident on the node $x_{ij}$, out of which $B(x_{ij})$ paths satisfy all the constraints as per the definition of CSF in \eqref{cdo_9}.

\begin{lemm}\label{lemma4}
The minimum number of maximum paths $N_e$ that need to be maintained in \eqref{cdo_16} at each stage in the recursion for optimality of $H$ can be expressed as a lower bound: 
\end{lemm}
\begin{equation}\label{cdo_lem4_3}
N_e \ge \underbrace{\text{max}}_{\substack i} \sum_{j=1}^M \sum_{k=1}^{L_{ij}} A(\bold{x}^{ij}_k).
\end{equation}

\section{Multi-survivor Dynamic Programming Algorithm}
\label{Algo}
\begin{algorithm}[h!]
  \caption{Multi-survivor Dynamic Programming.}\label{AlgoMVA}
  \begin{algorithmic}
    \small
    \State {\bfseries Input:}{$M$,$N$,$A(\cdot)$,$\mathcal{X}$}
      \State $M\gets \text{Cardinality of sets }\mathcal{X}$
      \State $N\gets \text{Number of stages or variables of }f(\cdot)\text{ of problem H}$
      \State $A(\cdot) \gets \text{Constraint satisfaction function of H}$
      \State $\mathcal{X} \gets \text{Set of discrete values that }x_i\text{'s can take }$
      \State \text{Determine the path valency $E(x_{ir})$ for all nodes} 
      \State \text{Initialize paths and its associated costs}
      \State $\scriptstyle \bold{x}^{1,1} = \{\mathcal{X}\}_1, \bold{x}^{1,2} = \{\mathcal{X}\}_2, \cdots, \bold{x}^{1,M} = \{\mathcal{X}\}_M$
      \State $\scriptstyle \lambda(\bold{x}^{1,1})=b_1\phi_1(\{\mathcal{X}\}_1), \lambda(\bold{x}^{1,2})=b_1\phi_1(\{\mathcal{X}\}_2), \cdots,$  
      \State $\scriptstyle \lambda(\bold{x}^{1,M})=b_1\phi_1(\{\mathcal{X}\}_M)$
      \For{$i=2:N$}
           \For{$r=1:M$}
           \State{$m \gets 1$}
           	\For{$b=1:M$}
		   \If {$(x_{{i-1}b},x_{ir}) \in T(A)$}
			\For{$k=1:E(x_{ir})$}
			   \If {$A(\bold{x}_k^{ir}) > 0$}
			       \State{$x_{ir} =  \{\mathcal{X}\}_r$}
			       \State{$\bold{x}^{i,r}_k = [(\bold{x}^{i-1,b}_k)^T, x_{ir}]^T$}
    			       \State{$\lambda(\bold{x}^{i,r}_k) = \lambda(\bold{x}^{i-1,b}_k) + b_i\phi_i(x_{ir})$}
			   \EndIf
			\EndFor
		   \EndIf
		\EndFor
		\State{\text{Keep only $\{\bold{x}_m^{i,r}\}_{m=1,\cdots,E(x_{i,r})}$ paths, that }
		\State{$\text{ correspond to }E(x_{ir})$ maximum values of} 
		\State{$\scriptstyle \Big\{\lambda(\bold{x}_m^{i,r})\Big\}_{m=1,\cdots,E(x_{ir})}$ and ignore the rest!}}
           \EndFor
     \EndFor
     \State{Process the termial node "F"}
     \State{$m \gets 1$}
     \For{$b=1:M$}
		\For{$k=1:L_{T}\text{ All paths incident on terminal node F}$}
			\State{$\scriptstyle \lambda(\bold{x}_m^F) = \lambda(\bold{x}^{N-1,b}_k)+b_N\phi_N(x_{Nb})$}
			\State{$\scriptstyle m \gets m+1$}
		\EndFor
     \EndFor
     \State{\text{Pick $\bold{x}^*$ that has maximum of $\Big\{\lambda(\bold{x}^F_k)\Big\}_{k=1,\cdots,m-1}$ }}
     \State{$\textbf{return }\bold{x}^* \text{ Optimal Solution}$}
  \end{algorithmic}
\end{algorithm}

\subsection{Evaluation of CSF}\label{const}
It can be shown that the $A(\bold{x}^m)$ can be evaluated at any intermediate stage $m \le N$ by posing it as an unconstrained discrete optimization (UDO) problem. The UDO problems can be solved in polynomial time using graph structures like Depth-First-Search (DFS), Breadth-First-Search (BFS), or BB algorithms \cite{Bbound, Ndale}.

\subsection{Computational Complexity Analysis}
\label{CCA}
The computational blocks of the msDP involve (i) evaluating the CSF $A(x^{ij}_k)$, and (ii) performing the ACMS operation on the $E(x_{ij})$ paths incident on a given node $x_{ij}$ of the trellis. For (i), a BFS or DFS can be used in such cases which has complexity of $O(E + V)$, where $E$ and $V$ are the numbers of edges and vertices of the trellis \cite{Ndale}. In the worst case of a fully connected trellis, $E = (N-1)M^2,$ and $V = NM$. Also, in the worst case scenario of $N_e^{'} = \frac{N_e}{M}$ incident on all the nodes in the given trellis, the net complexity of CSF evaluation $T_{CSF}$ can be expressed as
\begin{equation}\label{cdo_cca1}
T_{CSF} = N_e\Big((N-1)NM^2+N^2M\Big)\text{ for all the }NM\text{ nodes}.
\end{equation}
In case of (ii), the total number of ACMS evaluations for the trellis in the worst case is $T_{ACMS} = \frac{N_e}{M}NM = N_eN$. Thus, the complexity of msDP is $O(N_eN^2M^2)$ or in general $O(N_eN^{l})$, where $l \in \mathbb{R}$. On the other hand, the ES requires an evaluation of the CSF and OF for $M^N$ solutions. Hence its complexity is $O(M^N)$. In the case of learning techniques like RL, to guarantee an optimal solution, the complexity required is $\approx O(M^N)$ \cite{Watkins}.

\section{Simulations} 
\label{Sim}
In this section, we apply the proposed msDP to solve (i) Analog to Digital Converter (ADC) Bit-Allocation (BA) problem for MaMIMO receiver, and (ii) DNA fragment assembly (DFA) problem that is a part of DNA sequencing \cite{Maria,DFA2}. We evaluate the performance and compare the number of computations\footnote{Number of computations = Actual number of CSF evaluations + ACMS operations.} against the ES and the heuristic suboptimal simulated annealing (SA) algorithms\footnote{The optimality is compromised against complexity with heuristic algorithms like SA.} \cite{SimAn2}.

\subsection{ADC Bit Allocation for MaMIMO}
The ADC BA problem is to assign the number of bits to be used by Variable-Resolution ADCs on different Radio Frequency (RF) paths of the MaMIMO receivers. An optimal BA ensures that the performance, of the receiver is maximized under a non-linear power constraint. It is to be noted that the OF is non-linear. In \cite{Zakir2}, the authors reduce this to a problem in $H$, which is described as
\begin{equation}\label{cdo_sim_ba1}
J = \underbrace{\text{max}}_{\substack{\{x_i\}_{i=1}^N;A(\bold{x}) > 0}}\Big\{\sum_{i=1}^{N} \frac{a_i^2}{b_i^2+d_i2^{x_i}}\Big\},
\end{equation}
where $a_i$, $b_i$, and $d_i$ are constants $\in \mathbb{R}$ that represent channel singular value, noise power, and coefficient of quantization noise due to bit allocation $x_i$ on the $i^{th}$ RF path, respectively. Here $N$ is the number of RF paths in the receiver. The CSF $A(\bold{x}) > 0$ iff the power constraint $\sum_{i=1}^{N} 2^{x_i} \le P_t$, and bit-ordering  constraint $x_1 \ge x_2  \ge \cdots  \ge x_N$ are satisfied. The total ADC power budget is $P_t$. Hence we have
\begin{equation}\label{cdo_sim3}
    A(\bold{x}) = 
\begin{cases}
    1,\text{ if } &\sum_{i=1}^{N} 2^{x_i} \le P_t,\\
       &x_1 \ge x_2  \ge \cdots  \ge x_N.\\
    0, &\text{ Otherwise}.\\
\end{cases}
\end{equation}
We perform simulations considering the parameters used by the authors in \cite{Zakir2}. They consider a maximum bit resolution of $M=4$, such that $x_i$'s (bits) take values from the set $\mathcal{X} = \{1,2,3,4\}$ for $1 \le i \le N$. We assume the number of RF paths in the receiver as $N=12$ and the power budget of $P_t=48$. For this case, we achieve $N_e \ge 13$. The number of computations with msDP, ES, and SA algorithms are shown in Table \ref{tab2}. Both the brute-force and msDP methods obtain the optimal solution of $\bold{x}^* = [4,2,2,2,2,2,1,1,1,1,1,1]^T$ \cite{Zakir5}.
 
\subsection{DNA Fragment Assembly}
We apply our proposed formalism to the DFA problem, which has equivalence to the TSP \cite{Maria, boundTSP}. DFA is the challenging process of the DNA sequencing \cite{Jay}. DNA sequencing's main problem is that the current technology can not read an entire genome in one shot, sometimes not even more than 1000 bases. Even the simplest organisms, like bacteria and viruses, have much longer genome lengths. Consequently, the genomes are broken down into smaller readable fragments and sequenced \cite{Sanger}. In this step, $N$ copies of DNA are created.  A short fragment is derived from each of the replicated DNA at some random location. These short fragments are sequenced. The final and challenging step is to assemble these sequenced fragments to obtain the original DNA sequence. This step is called the DFA, which is illustrated through an example below \cite{Maria}.\\
\indent We assume the DNA sequence to be $TTACCGTGC$, and the fragments sequenced using 4 DNA copies being $F_1 = ACCGT$, $F_2 = CGTGC$, $F_3 = TTAC$, and $F_4 = TACCGT$. The overlap of each fragment with the other three fragments are computed using similarity measure. Based on this similarity measure, the order of fragments are determined which in the case of this example is $F_3F_4F_1F_2$.\\
\indent The DFA problem is posed as a maximization problem where the sum of the similarity measures\footnote{The similarity scores between different pairs of fragments are derived using the Smith-Waterman algorithm \cite{Maria, Smith}.} between two adjacent DNA fragments is maximized. This is subject to the constraint that there is no repetition of the fragments in the sequence. Formally, this problem is defined as \cite{Maria}
\begin{equation}\label{cdo_sim1}
J = \underbrace{\text{max}}_{\substack{\{F_{\sigma_i}\}_{i=1}^N;A(\mathcal{F}) > 0}}\Big\{\sum_{i=1}^{N-1} \phi(F_{\sigma_i},F_{\sigma_{i+1}})\Big\},
\end{equation}
where $\mathcal{F} = \{ F_{\sigma_1}, F_{\sigma_2},\cdots,F_{\sigma_N} \}$ is the optimal solution indicating the original DNA sequence, $\sigma_i$ is the fragment index, and $\phi(F_{\sigma_i},F_{\sigma_{i+1}})$ is the similarity measure between the fragments $F_{\sigma_i}$ and $F_{\sigma_{i+1}}$. Here, the set $\mathcal{X}$ is collection of DNA fragments $\{F_{\sigma_j}\}_{j=1}^{N}$, where $N$ is the number of fragments. In this example we have $M=N$. For this problem the CSF $A:\mathcal{F} \rightarrow \{0,1\}$ is defined on the set $\mathcal{F}$ as $A(\mathcal{F}) > 0$ iff all the fragments in $\mathcal{F}$ are unique. The CSF denoted as
\begin{equation}\label{cdo_sim3}
    A(\mathcal{F}) = 
\begin{cases}
    1,&\text{ if }\bigcap\limits_{i=1}^{N} F_{\sigma_i} = \varnothing,\\ 
    0, &\text{ Otherwise}.\\
\end{cases}
\end{equation} 
We perform the DFA of a small section of the DNA sequence of bacterium \textit{Escherihia Coli (E. coli)} \cite{Maria}. The original section of the DNA is represented as  $TACTAGCAATACGCTTGCGTTCGGT$. We consider $N=10$ fragments each with 8 bases, as follows: $F_1 = \textit{ACGCTTGC}$, $F_2 = \textit{TTGCGTTC}$, $F_3 = \textit{ACTAGCAA}$, $F_4 = \textit{CGTTCGGT}$, $F_5 = \textit{AGCAATAC}$, $F_6 = \textit{TACTAGCA}$, $F_7 = \textit{AATACGCT}$, $F_8 = \textit{CTTGCGTT}$, $F_9 = \textit{ATACGCTT}$, and $F_{10} = \textit{CTAGCAAT}$.
Using the proposed method, we obtain the optimal solution $\mathcal{F} = \{F_6 F_3 F_{10} F_5 F_7 F_9 F_1 F_8 F_2 F_4 \}$. It can be shown that by imposing a second constraint using an upper bound \cite{boundTSP}, we achieve $N_e \ge 454$. Solving the DFA problem to optimality with the imposition of these additional constraints is not possible using the traditional DP. The results are shown in Table \ref{tab3}.

\begin{table}[t]
\caption{Comparison of the total actual number of computations $^\mathsection$ required for msDP, ES, and SA algorithms.}
\label{tab2}
\begin{center}
\begin{sc}
\resizebox{\columnwidth}{!}{%
\begin{tabular}{ |c|c|c|c|}
\hline
& \small Proposed & \small Exhaustive & \small Simulated \\
& \small Method  & \small Search & \small annealing \\ 
& \tiny $O(N_eN^2M^2)$  & \tiny $O(M^N),O(N!)$ & \tiny $\approx O(M^N),O(N!)^{\ddagger}$ \\
\hline
\small ADC BA  & \small 6912 $(N_e = 13)$ & \small $16.7 \times 10^6$ & \small 5220 \\
\hline
\small DFA-1{$^{\dagger}$} & \small 83890 $(N_e = 434)$ & \small $3.62 \times10^6$ & \small 380 \\
\hline
\end{tabular}%
}
\end{sc}
\tiny{$^{\dagger}$} DFA with a second constraint. \tiny{$\mathsection$ Number of computations = Actual number of CSF evaluations + ACMS operations}. \tiny{$^{\ddagger}$} To guarantee optimal solution.
\end{center}
\vskip -0.1in
\end{table}

\begin{table}[h!]
\caption{Comparison of performance of the msDP, ES, and SA algorithms.}
\label{tab3}
\begin{center}
\begin{sc}
\resizebox{\columnwidth}{!}{%
\begin{tabular}{ |c|c|c|c|}
\hline
& \small Proposed & \small Exhaustive & \small Simulated \\
& \small Algorithm  & \small Search & \small annealing\\ 
& \small \textbf{\textcolor{blue}{Optimal}}  & \small \textbf{\textcolor{blue}{Optimal}} & \small \textbf{\textcolor{red}{Suboptimal}} \\ 
\hline
\small ADC BA  & \small \textcolor{blue}{$[4,2,2,2,2,2,1,1,1,1,1,1]^T$}  & \small \textcolor{blue}{$[4,2,2,2,2,2,1,1,1,1,1,1]^T$} & \small \textcolor{red}{$[4,4,2,2,1,1,1,1,1,1,1,1]^T$} \\
\hline
\small DFA-1 \small{$^{\dagger}$} & \small \textcolor{blue}{$[6,3,10,5,7,9,1,8,2,4]^T$} & \small \textcolor{blue}{$[6,3,10,5,7,9,1,8,2,4]^T$} & \small \textcolor{red}{$[6,3,10,5,2,8,4,7,9,1]^T$} \\
\hline
\end{tabular}%
}
\end{sc}
\end{center}
\vskip -0.1in
\end{table}

\section{Conclusion} 
\label{Conc}
Constrained discrete optimization problems with the objective function that satisfies the optimality principle and place no conditions on the constraints are NP-Hard. In general, such problems arise in many areas of science, engineering, finance, and management. We propose a multi-survivor dynamic programming framework that solves these problems to optimality with worst-case complexity reduction from either $O(K^N)$ or $O(N!)$ to $O(N_eN^l)$, where $N_e$ and $l$ are constants that depend on the constraints.

\section*{Acknowledgment}
The authors would like to thank National Instruments for the financial support extended for this work.  
\bibliographystyle{IEEEtran}
\bibliography{LatCom_BibTexFile}
%
\end{document}